\documentclass[12pt,reqno]{amsart} 

\usepackage{
	graphicx,
	color,
	amsmath,
  csquotes
}
\usepackage[numbers,sort&compress]{natbib} 

\expandafter\let\csname ver@amsthm.sty\endcsname\relax
\let\theoremstyle\relax

\usepackage{hyperref}
\usepackage{amsthm}
\usepackage{cleveref}

\newcommand\nc\newcommand
\nc\rc\renewcommand
\nc\dc\DeclareMathOperator

\nc\nt\newtheorem

	\theoremstyle{plain}
		\nt{thm}{Theorem}
		\nt*{thm*}{Theorem}
		\nt{prop}[thm]{Proposition}
		\nt{lemma}[thm]{Lemma}
		\nt{cor}[thm]{Corollary}
		\nt{claim}[thm]{Claim}
    \nt{algorithm}{Algorithm}
    \nt{question}{Question}
	\theoremstyle{definition}
		\nt{dfn}{Definition}
	\theoremstyle{remark}
		\nt*{rmk}{Remark}

	\nc\crn\crefname

	\crn{thm}{Theorem}{Theorems}
	\crn{lemma}{Lemma}{Lemmas}
	\crn{prop}{Proposition}{Propositions}
	\crn{cor}{Corollary}{Corollaries}
	\crn{section}{Section}{Sections}
  \crn{subsection}{Subsection}{Subsections}
	\crn{figure}{Figure}{Figures}
	\crn{claim}{Claim}{Claims}
	\crn{dfn}{Definition}{Definitions}
	\crn{rmk}{Remark}{Remarks}
  \crn{algorithm}{Algorithm}{Algorithms}
  \crn{question}{Question}{Questions}

\dc{\id}{id}
\dc{\Tr}{Tr}
\dc\len{len}
\dc\rank{rank}
\dc\rec{rec}
\dc\SCC{SCC}
\dc\NotMin{NotMin}
\dc\NotMax{NotMax}
\dc\Min{Min}
\dc\Max{Max}

\nc\m\mathbb
\rc\c\mathcal
\nc{\f}{\frac}
\nc{\one}[1]{\mathbf{1}\!\left[#1\right]}
\nc{\fl}[1]{\left\lfloor #1 \right\rfloor}
\nc{\ceil}[1]{\left\lceil #1 \right\rceil}
\nc{\al}[1]{\begin{align*}#1\end{align*}}
\nc{\tx}{\texorpdfstring}
\nc{\ds}{\displaystyle}

\rc{\hat}{\widehat}
\rc{\theenumi}{\roman{enumi}}

\title{Finitely dependent insertion processes}

\author{Avi Levy}
\address{Avi Levy, Department of Mathematics,
University of Washington, Seattle, WA 98195, USA}
\email{avius at uw.edu}
\urladdr{\url{http://www.math.washington.edu/~avius}}

\date{\today}

\begin{document}

\begin{abstract}
	A $q$-coloring of $\mathbb Z$ is a random process assigning one of $q$ colors to
  each integer in such a way that consecutive integers receive distinct colors.
  A process is $k$-dependent if any two sets of integers separated by a distance  greater than $k$ receive independent colorings. Holroyd and Liggett
  constructed the first stationary $k$-dependent $q$-colorings by introducing
  an insertion algorithm on the complete graph $K_q$. We extend their
  construction from complete graphs to weighted directed graphs.
  We show that complete multipartite analogues of $K_3$ and $K_4$ are the only graphs whose insertion process is finitely dependent and whose insertion algorithm is consistent. In particular, there are no other such graphs among all unweighted graphs and among all loopless complete weighted directed graphs. Similar results hold if the consistency condition is weakened to eventual consistency. Finally we show that the directed de Bruijn graphs of shifts of finite type do not yield $k$-dependent insertion processes, assuming eventual consistency.
	\vspace{-5ex}
\end{abstract}

\maketitle

\section{Introduction}
	A {\bf proper $q$-coloring of $\m Z$} is a sequence of colors
  $(x_i)_{i\in \m Z}$ with $x_i\in [q]:=\{1,\ldots,q\}$ such that
  $x_i\not=x_{i+1}$ for all $i$. A random
  $q$-coloring
  $(X_i)_{i\in \m Z}$ is {\bf stationary} if $(X_i)_{i\in \m Z}$ and
  $(X_{i+1})_{i\in\m Z}$ are equal in law. A stationary $q$-coloring is
  {\bf $k$-dependent} if $(X_i)_{i< 0}$ and $(X_i)_{i\geq k}$ are independent,
  and {\bf finitely dependent} if it is $k$-dependent for some $k\geq 0$.

  The simplest examples of stationary finitely dependent processes are the
  {\bf block factors}. These are stochastic processes of the form
  $\{f(Y_i,\ldots, Y_{i+k})\}_{i\in\m Z}$ where $f$ is deterministic and
  $\{Y_i\}_{i\in \m Z}$ are an i.i.d. sequence. In the 1960s, Ibragimov and
  Linnik first suggested that there may exist non-block factor stationary
  finitely dependent processes \cite{ibragimov1965independent,MR0322926}.
  Since then, examples of such processes have been constructed by several authors
  in the course of studying properties of finitely dependent processes
  \cite{MR1176437,burton1993,MR1269540,MR972778,MR744235}. Until
  recently, it has been believed that most `natural' finitely dependent
  processes are block factors \cite{MR2721041}.

  Yet block factors have subtle limitations. For example, these processes are
  never supported on proper colorings \cite{MR3262072}. It turns out that
  finitely dependent processes do not have this limitation, although this fact
  is highly non-obvious and remains to be fully understood. This was discovered
  by Holroyd and Liggett in a recent breakthrough \cite{holroyd60}, in which
  they disproved a conjecture
  of Schramm \cite{holroyd64} by showing that stationary finitely dependent
  colorings of
  the integers exist. These are perhaps the first natural non-block
  factor finitely dependent processes.

  Specifically, Holroyd and Liggett constructed symmetric $3$- and
  $4$-colorings with these properties. It is remarkable that, while their construction produces a $q$-coloring for each integer $q\geq 2$, only when $q\in\{3,4\}$ is the coloring finitely dependent. Using a more complicated construction, these authors later obtained symmetric $q$-colorings for all $q\geq 4$ \cite{holroyd65}.


  As described in \cite{holroyd60}, the $q$-colorings therein have the following characterization. For each integer $q\geq 2$, let $Z=(Z_1,\ldots,Z_n)$ be a sequence of independent random variables each taking the values $1,2,\ldots,q$ with equal probability. Let $\sigma$ be an independent uniformly random permutation of $1,\ldots,n$, which we interpret as meaning that the symbol $Z_{\sigma(i)}$ arrives at time $i$. Let $E$ be the event that, for every time $t=1,\ldots,n$, the subsequence of $Z$ formed by those symbols that arrived up to time $t$ (ordered as in the original sequence $Z$) forms a proper coloring (i.e. no two consecutive elements in the subsequence are equal). The conditional law of $Z$ given $E$ equals the law of $(X_1,\ldots,X_n)$, where $X$ is the $q$-coloring constructed in \cite{holroyd60}.

  It was observed by Holroyd (personal communication) that the proper coloring condition in the previous paragraph may be replaced by a graph adjacency condition. The case of a $q$-coloring corresponds to the complete graph with vertex set $\{1,\ldots,q\}$, denoted $K_q$. A general graph will encode which pairs of vertices may appear consecutively. See e.g. \cite[Example 2.5]{borgs2006counting} for more on this perspective. Since few stationary finitely dependent colorings are currently known, it is natural to pursue this generalization in the search for new finitely dependent processes.
  
  Fix a finite graph $G$ containing at least one edge. Let $Z=(Z_1,\ldots,Z_n)$ be a sequence of independent uniformly random vertices of $G$. Let $\sigma$ be an independent uniformly random permutation of $1,\ldots,n$, which we interpret as meaning that the vertex $Z_{\sigma(i)}$ arrives at time $i$. Let $E$ be the event that, for every time $t=1,\ldots,n$, the subsequence of $Z$ formed by those vertices that arrived up to time $t$ (ordered as in the original sequence $Z$) forms a path in $G$. Finally, let $(Y_1,\ldots,Y_n)$ denote random variables whose law equals the conditional law of $Z$ given $E$.


  To produce a stochastic process $(X_i)_{i\in\m Z}$ from these finite marginals, we must apply a limiting procedure. Let $P_n$ denote the probability mass function of $(Y_1,\ldots,Y_n)$. If for all $n$ the mass functions of $(Y_1,\ldots,Y_{n-1})$ and $(Y_2,\ldots,Y_n)$ equal $P_{n-1}$, then by the Kolmogorov Extension Theorem \cite{PTE} there exists a unique stochastic process $(X_i)_{i\in\m Z}$ such that $(X_1,\ldots,X_n)$ has mass function $P_n$ for all $n\geq 1$. In this case we say that $\{P_n\}_{n=1}^{\infty}$ is consistent and that $G$ satisfies property (C). We say the marginals are eventually consistent if the preceding condition holds for all sufficiently large $n$, in which case $G$ satisfies property (EC). In this case, we construct a process $(X_i)_{i\in\m Z}$ in the same manner as before by taking a projective limit \cite{MR1876169} over $n$ sufficiently large, generalizing the previous construction when $G$ has property (C). We call $(X_i)_{i\in\m Z}$ the {\bf insertion process} associated to $G$.

  Holroyd and Liggett showed that for all $q\geq 2$, the complete graph $K_q$
  has property (C) \cite[Proposition 10]{holroyd60}. Furthermore, they
  discovered that the insertion process associated to $K_4$ is $1$-dependent
  and the insertion process associated to $K_3$ is $2$-dependent. Moreover,
  these are the \emph{only} values of $q$ for which the process is finitely
  dependent \cite[Proposition 13]{holroyd60}.

  One may embellish these examples in our general setting. Replacing each of
  the $3$ (resp. $4$) colors with $r$ copies of itself yields the complete
  multipartite graphs $K_{r,r,r,r}$ (resp. $K_{r,r,r}$). Both of these graphs
  are easily seen to have property (C) and have a $1$- (resp. $2$-)dependent
  insertion process.
  We establish the remarkable fact that these are essentially the \emph{only} (EC) graphs
  with a finitely dependent insertion process.

	\begin{thm}\label{thm:main}
		Consider a finite graph $G$ having property (C). Then the insertion
    process associated to $G$ is $k$-dependent if and only if either:
    $G=K_{r,r,r}$ and $k\geq 2$, or $G=K_{r,r,r,r}$ and $k\geq 1$.

    If instead we assume that $G$ has property (EC), then the same characterization holds except $G$ may in addition be a disjoint union of one of the above graphs with a collection of isolated vertices.
	\end{thm}

  We deduce \cref{thm:main} from \cref{thm:second}, which is a result about weighted graphs. In a {\bf weighted digraph}, an edge from vertex $i$ to vertex $j$ has weight $w(i,j)$. There is no assumption that $w(i,j)=w(j,i)$, and a weight of 0 signifies that no edge is present. There is a natural extension of the insertion process to the setting of weighted digraphs, provided that the digraph satisfies a natural analogue of property (EC)\footnote{See \cref{section:weighted}; the only difference is that instead of conditioning on the sequence of vertices forming a path, a bias is applied depending on certain edge weights.}. In lieu of a precise description of these properties for weighted digraphs, we present the following iterative sampling algorithm for the insertion procedure on weighted digraphs.

  \begin{algorithm}\
  \label{alg:wGraphInsert}
  \begin{enumerate}
    \item Let $P_1$ assign uniform weight to all $v\in V$.

    \item Given a random $X\in V^n$ with mass function $P_n$, sample $I\in\{0,\ldots,n\}$ and $\c V\in V$ with probability proportional to
    $$
      \m P(I=i,\c V=v)\propto w(x_{i},v)\cdot w(v,x_{i+1}),
    $$
    where the first or second factor is left out when $i=0$ or $i=n$,  respectively.

    \item Output the sequence $(X_1,\ldots,X_I,\c V,X_{I+1},\ldots,X_n)$. It has mass function $P_{n+1}$.
  \end{enumerate}
  \end{algorithm}

  A weighted digraph is {\bf uniform of weight $w$} if $w(i,j)\in \{0,w\}$ for all vertices $i$ and $j$, and if in addition $w(i,j)=w(j,i)$ and $w(i,i)=0$. We remark that uniform weight graphs underly a model of endpoint-weighted insertion introduced in \cite[\S 4]{holroyd63}. 

  \begin{thm}\label{thm:second}
    Consider a finite uniform weight graph $G$ having property (C). Then the insertion
    process associated to $G$ is $k$-dependent if and only if either:
    $G=K_{r,r,r}$ and $k\geq 2$, or $G=K_{r,r,r,r}$ and $k\geq 1$.

    If instead we assume that $G$ has property (EC), then the same characterization holds except $G$ may in addition be a disjoint union of one of the above graphs with a collection of isolated vertices.    
  \end{thm}

	\cref{thm:main,thm:second} demonstrate that the requirement of satisfying 
  property (EC) and having a finitely dependent associated insertion process is extremely restrictive on a graph. This trend continues beyond the setting of uniform weight graphs.

	\begin{thm}\label{thm:third}
		Let $G$ be a weighted digraph satisfying property (C) such that $w(i,i)=0$ and $w(i,j)>0$ for all distinct vertices $i$
    and $j$. Then the insertion process associated to $G$ is $k$-dependent
    if and only if $G$ is unweighted and either: $G=K_3$ and $k\geq 2$;
    or $G=K_4$ and $k\geq 1$.
	\end{thm}

  As described previously, we are considering generalizations of a family of colorings constructed in \cite{holroyd60} in the search for new finitely dependent processes. We have considered more general graph constraints in place of colorings. However, there is an even broader type of local constraint to consider, known as a shift of finite type \cite{MR1369092,holroyd64}. A {\bf loopless shift of finite type} is a set
  of the form
	\[
		S=\Big\{x\in [q]^{\m Z}\colon (x_{i+1},\ldots,x_{i+n})\in W\quad
    \forall i\in \m Z\Big\},
	\]
	for some $W\subset [q]^n\setminus \{(i,\ldots,i)\colon i\in [q]\}$. The
  {\bf de Bruijn graph} \cite{MR0018142} associated to $S$ has vertex set
  $W$ and edge set
	\[
		E=\Bigl\{\bigl((x_1,\ldots,x_n),(x_2,\ldots,x_{n+1})\bigr)\colon
    (x_1,\ldots,x_n),(x_2,\ldots,x_{n+1})\in W\Bigr\}.
	\]

	\begin{samepage}
	\begin{thm}\label{thm:shiftFinite}
		Consider a loopless shift of finite type with de Bruijn graph $G$.
    Suppose that $G$ has property (EC). Then the insertion process
    associated to $G$ extends to a process supported on the shift of finite
    type. Moreover this process is not finitely dependent.
	\end{thm}
	\end{samepage}

	When $n=2$, shifts of finite type correspond to edge sets of directed
  graphs. In analogy with \cref{thm:main,thm:second,thm:third}, one might
  expect $K_3$ and $K_4$ to appear in \cref{thm:shiftFinite}. The reason this is not the case is that 
  \cref{thm:main,thm:second,thm:third} pertain to insertion processes on the vertex set, whereas \cref{thm:shiftFinite} corresponds to insertion processes on the edge set.

  \subsection*{Open questions}

    We have presented a general algorithm that produces a sequence of mass functions on paths of growing length. One must impose some condition on these mass functions in order to obtain a limiting process on $\m Z$. We have considered two such conditions: property (C), and property (EC), corresponding to consistency and eventual consistency of the mass functions, respectively.

    But there are other ways to obtain a limiting process on bi-infinite paths. For example, one may use weak limits.

    \begin{question}\label{question:1}
      Does \cref{thm:main} continue to hold if we replace property (EC) with the assumption that the measures induced by the mass functions $P_n$ (constructed above) converge weakly?
    \end{question}

    In \cref{thm:third}, we do not even know if property (C) can be weakened to property (EC), let alone to weak limits.

    \begin{question}\label{question:2}
      Does \cref{thm:third} continue to hold for strongly connected weighted digraphs if we replace property (C) with property (EC)?
    \end{question}

    While we have established several results in the general setting of weighted digraphs, there remains a relatively simpler context that remains unresolved.
    \begin{question}\label{question:3}
      Does \cref{thm:main} continue to hold for strongly connected unweighted digraphs?
    \end{question}

    Lastly, recall that in \cref{thm:third} we must assume strict positivity of the edge weights, which is an open condition (in the topological sense). It may come as a surprise that there does not appear to be a simple limiting argument that allows one to remove this assumption. One way to see this is that the conclusions of the theorem must change when one removes these positivity conditions, due to complete multipartite graphs with more than two vertices in each of the partite sets.
    \begin{question}\label{question:4}
      If we remove the assumption that $w(i,j)>0$ for all distinct vertices $i$ and $j$ but retain the hypotheses that property (C) holds and that the insertion process is $k$-dependent in \cref{thm:third}, does it follow that $G$ is unweighted and either $G=K_{r,r,r}$ and $k\geq 2$ or $G=K_{r,r,r,r}$ and $k\geq 1$?
    \end{question}

	\subsection*{Overview}
	\Cref{section:weighted} presents the main construction underlying the remainder of the paper, which consists of a sequence of mass functions on paths of growing length in a weighted digraph. \cref{thm:main,thm:second} are proven in \Cref{section:unifWeight}. Uniform
  weight graphs appear in Subsection~\ref{subsection:propUnif},
  complete multipartite graphs in Subsection~\ref{subsection:modClique},
  and the proofs of \cref{thm:main,thm:second} are in
  Subsection~\ref{subsection:uniqueUnif}. \Cref{section:arbWeight} combines
  \cref{thm:second} with some additional arguments to deduce \cref{thm:third}.
  Lastly, we deduce \cref{thm:shiftFinite} in \Cref{section:shifts} from \Cref{thm:lackTri} in Subsection~\ref{section:kDep}.

\section{Weighted Insertion}\label{section:weighted}
  In this section we introduce a generalization of the construction in \cite[\S 2]{holroyd65} to weighted digraphs.


    Let $V$ be a finite alphabet. A {\bf word} (of length $n$) is a finite sequence $x=(x_1,x_2,\ldots,x_n)\in V^n$, which we sometimes abbreviate to $x_1x_2\cdots x_n$. The word of length $0$ is denoted by $\emptyset$. Let $S_n$ be the symmetric group of all permutations of $1,\ldots,n$. Let $x\in V^n$ be a word and let $\sigma\in S_n$ be a permutation. We interpret $\sigma$ as meaning that at time $t=1,\ldots,n$ symbol $x_{\sigma(t)}$ arrives in (relative) position $\sigma(t)$.
    
    Let $\Min(\sigma)$ and $\Max(\sigma)$ denote the sets of running minima and maxima of $\sigma$, respectively. For example, $\Min(\sigma)$ is given by
    $$
      \Min(\sigma):=\{1\leq t\leq n\colon \sigma(t)=\min_{s\leq t}\sigma(s)\}.
    $$
    If $\sigma(t)$ is neither a running minimum nor maximum, then at time $t$ the symbol $x_{\sigma(t)}$ is inserted between $x_{\sigma(t^-)}$ and $x_{\sigma(t^+)}$ where
    $$
      \sigma(t^-)=\max_{s<t}\{\sigma(s)\colon \sigma(s)<\sigma(t)\}\text{ and }\sigma(t^+)=\min_{s<t}\{\sigma(s)\colon \sigma(s)>\sigma(t)\}.
    $$

    A {\bf weighted digraph} with vertex set $V$ is a function $w\colon V^2\to\m R_{\geq 0}$. 
    We define the {\bf weight} of the pair $(x,\sigma)\in V^n\times S_n$ to be 
    \begin{equation}\label{eq:weight}
      w(x;\sigma):=\prod_{t\not\in \Min(\sigma)}w\bigl(x_{\sigma(t^-)},x_{\sigma(t)}\bigr)\prod_{t\not\in\Max(\sigma)}w\bigl(x_{\sigma(t)},x_{\sigma(t^+)}\bigr).
    \end{equation}
    For example, when $\sigma=\id$ is the identity permutation of $1,\ldots,n$ we have $w(x;\id)=w(x_1,x_2)\cdots w(x_{n-1},x_n)$. We denote this quantity by $w(x)$. When the length of $x$ is at most 1, we have that $w(x)=1$.

    If we imagine building the word dynamically using $\sigma$, then when $x_{\sigma(t)}$ is inserted between $x_{\sigma(t^-)}$ and $x_{\sigma(t^+)}$ a multiplicative weight of $w(x_{\sigma(t^-)},x_{\sigma(t)})$ $w(x_{\sigma(t)},x_{\sigma(t^+)})$ is incurred in \eqref{eq:weight}.

	\begin{figure}
		\centering
		\def\svgwidth{\textwidth}
		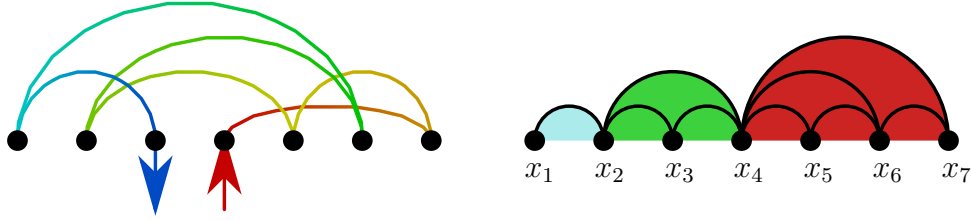
		\caption{Left: The arrival order for $\sigma=4752613$. Right: Multiply the edge weights to get $w(x_1\cdots x_7;\sigma)$.
    \label{fig:bubble}}
	\end{figure}

	\begin{dfn}\label{dfn:B}
		Given a weighted digraph with vertex set $V$ and a word $x\in V^n$, define \[B(x):=\sum_{\sigma\in S_n}w(x;\sigma).\]
	\end{dfn}
	This is a generalization of the building number defined in \cite{holroyd60}, which is the special case consisting of the weighted digraph $w$ with vertex set $\{1,\ldots,q\}$ and weight function $w(i,j)=\one{i\not=j}$.

  Clearly $B(x)>0$ if and only if there exists $\sigma\in S_n$ such that $w(x;\sigma)>0$. In this case $\sigma=\id$ works. Word $x$ has {\bf positive weight} if either of the equivalent conditions $B(x)>0$ or $w(x)>0$ holds.

  For a word $x=x_1x_2\cdots x_n$, we write $\hat x_i=x_1\cdots x_{i-1}x_{i+1}\cdots x_n$.
	\begin{lemma}\label{lem:rec}
	Suppose that $x\in V^n$. Then $B(x)$ is equal to
	\begin{equation*}
		B(\hat x_1)w(x_1,x_2)+\sum_{i=2}^{n-1} w(x_{i-1},x_i)\; B(\hat x_i)\; w(x_i,x_{i+1})+w(x_{n-1},x_n)B(\hat x_n).
	\end{equation*}
	\end{lemma}
	\begin{proof}
		Fix an index $i$. For any $\sigma$ which satisfies $i=\sigma(n)$,
		\[
			w(x;\sigma)=w(x_{i-1},x_i)\; w(\hat x_i;\widehat{\sigma}_i)\;
      w(x_i,x_{i+1}),
		\]
    with the obvious modifications when $i=1,n$. Summing over all $\sigma\in S_n$ yields the result.
	\end{proof}

	\subsection{Eventual Consistency}
  \label{subsection:weightedInsert}
    Let $G$ be a weighted digraph with (finite) vertex set $V$. Say that $G$ is {\bf recurrent} if there exist arbitrarily long words of positive weight. We remark that this is equivalent to there existing positive-weight words of \emph{every} length, since if $w(x_1x_2\cdots x_n)>0$ then also $w(x_1x_2\cdots x_k)>0$ for all $1\leq k\leq n$. Furthermore, it is easy to see that $G$ is recurrent if and only if 
    \begin{equation}\label{eq:posWeight}
      \sum_{y\in V^n} B(y)>0\qquad\text{ for all }n\geq 0.
    \end{equation}
    If $G$ is a recurrent digraph, then for all $n\geq 0$ we define a probability mass function $P_n$ on $V^n$ by setting
    \begin{equation}\label{eq:massDef}
      P_n(x)=\frac{B(x)}{\sum_{y\in V^n} B(y)},\qquad x\in V^n.
    \end{equation}    

    It is a simple consequence of \Cref{lem:rec} that \Cref{alg:wGraphInsert} presented in the introduction does in fact sample according to the mass function $P_n$ in \eqref{eq:massDef}. Moreover, \Cref{alg:wGraphInsert} is well-defined provided that $G$ is recurrent (otherwise, it terminates early if it reaches an $n$ such that there are no positive-weight words of length $n$).

    If $G$ is recurrent and the mass functions $\{P_n\}_{n=0}^{\infty}$ satisfy
		\begin{equation}\label{eq:consist}
			P_n(x)=\sum_{v\in V}P_{n+1}(xv)=\sum_{v\in V}P_{n+1}(vx)\qquad\text{ for all }n\geq 0,
		\end{equation}
    we say that $G$ satisfies {\bf property (C)} and that the mass functions are {\bf consistent}. Likewise if $G$ is recurrent and \eqref{eq:consist} holds for all sufficiently large $n$, we say that $G$ satisfies {\bf property (EC)} and that the mass functions are {\bf eventually consistent}.

    Clearly, equation \eqref{eq:consist} is equivalent to the existence of constants ${C_n>0}$ such that \eqref{eqn:propCdfn} holds for all $x\in V^n$:
    \begin{equation}\label{eqn:propCdfn}
      \sum_{v\in V}B(xv)=\sum_{v\in V}B(vx)=C_nB(x).
    \end{equation}
    Thus, a recurrent graph satisfies property (C) if and only if there exists $C_n>0$ such that \eqref{eqn:propCdfn} holds for all $n\geq 0$. Similarly, a recurrent graph satisfies property (EC) if and only if for all sufficiently large $n$, there exists $C_n>0$ such that \eqref{eqn:propCdfn} holds.





		When the associated mass functions are eventually consistent, the Kolmogorov
    extension theorem \cite{PTE} establishes that there exists a unique process $(X_i)_{i\in \m Z}$ with marginal distributions given by
		$$
			\m P\bigl((X_{i+1},\ldots,X_{i+n})=x\bigr)=P_n(x)
		$$
    for all sufficiently large $n$. This is the {\bf insertion process}
    associated to an (EC) weighted digraph. It follows by construction that the insertion process is stationary.


		\begin{lemma}\label{lem:kDep}
			Suppose that $G$ has property (EC). Then the associated insertion
      process is $k$-dependent if and only if for all $n,m\in\m N$
      sufficiently large, there are positive constants $C_{n,m}>0$ for which
			\begin{equation}\label{eq:factor}
				\sum_{W\in V^k}B(x\; W\; y)=C_{n,m}B(x)B(y),
        \qquad x\in V^n\text{ and }y\in V^m.
			\end{equation}
		\end{lemma}
		\begin{proof}
			By normalizing \eqref{eq:factor} we see it is equivalent to
			\begin{equation*}\label{eq:factor2}
				\sum_{W\in V^k}P_{n+k+m}(x\; W\; y)=P_n(x)P_m(y).
			\end{equation*}
      This is equivalent to the associated insertion process $(X_n)_{n\in\m Z}$ having the property that $(X_i)_{i\in I}$ and $(X_j)_{j\in J}$ are independent for any pair of sufficiently large intervals $I$ and $J$ separated by a distance of $k$ or greater. Since restriction preserves independence, we may remove the adjective `sufficiently large' from the previous sentence.
		\end{proof}

		Holroyd and Liggett \cite{holroyd60} proved that the (unweighted, undirected) complete
    graph $K_q$ has property (C) for all $q\geq 2$. Furthermore, they showed
    that the associated insertion process
    is $k$-dependent if and only if $q=3$ and $k\geq 2$ or $q=4$ and $k\geq 1$.
    In \Cref{section:unifWeight} we show that apart from minor modifications,
    these are the only unweighted and undirected graphs with property (EC) for which the associated insertion process is finitely dependent.

	\subsection{Necessary conditions for \tx{$k$}{k}-dependence}
  \label{section:kDep}
		Consider a weighted digraph $G$ satisfying property (EC). In this subsection, we present necessary conditions for the associated insertion process to be finitely dependent. We show that such digraphs must contain directed triangles (\cref{thm:lackTri}) and are not too far from being strongly connected (\cref{lem:restrictiveStructure}).

		A {\bf directed triangle on $(a,b)$} is a triple of vertices $(a,b,c)\in V^3$ that satisfy	$w(a,b)\; w(b,c)\; w(a,c)>0$. We consider weighted digraphs that lack directed triangles. There is a natural interpretation of this condition in relation with \cref{alg:wGraphInsert}: it is equivalent to requiring all insertions to occur at endpoints of the interval.
		\begin{lemma}\label{lem:lackTri}
			For a weighted digraph lacking directed triangles, any $x\in V^n$ satisfies $B(x)=2^{n-1}w(x)$.
		\end{lemma}
		\begin{proof}
			For any $1<i<n$, since $w(x_{i-1},x_i)>0$ and $w(x_i,x_{i+1})>0$ we must have that $w(x_{i-1},x_{i+1})=0$. Applying \cref{lem:rec} yields that
      \begin{equation*}
        B(x)=w(x_1,x_2)\cdot B(\hat x_1)+w(x_{\ell-1},x_\ell)\cdot B(\hat x_\ell),
      \end{equation*}
      which implies the claim by induction on the length of the word.
		\end{proof}

		\begin{lemma}\label{thm:lackTri}
      If a weighted digraph satisfies property (EC) and the associated
      insertion process is finitely dependent, then it contains a directed triangle.
		\end{lemma}

		\begin{proof}
			Suppose there was a weighted digraph $G$ lacking directed triangles that satisfies property (EC) and whose associated insertion process is $k$-dependent for some $k\geq 0$. Let $V$ denote its vertex set. Applying \cref{lem:kDep}, it follows that for all sufficiently large $m$ and $n$ and for all vertices $i,j\in V$ we have that
      \begin{equation}\label{eq:lem8sum}
        \sum_{P\in V^m,W\in V^k,Q\in V^m}B(iPWQj)=C_{m,n}\sum_{P\in V^m}B(i\; P)\sum_{Q\in V^n}B(Q\; j).
      \end{equation}
      Let $A$ denote the adjacency matrix of $G$, which is the $V\times V$ matrix with $A_{ij}=w(i,j)$. Using \cref{lem:lackTri} to express the left hand side of \eqref{eq:lem8sum} in terms of the adjacency matrix we obtain that
      \begin{equation}\label{eq:adjPow}
        (A^{m+k+n+1})_{ij}=2^{-m-k-n+1}\sum_{P\in V^m,W\in V^k,Q\in V^m}B(iPWQj).
      \end{equation}

      Combining \eqref{eq:lem8sum} with \eqref{eq:adjPow} implies that $\rank(A^N)\leq 1$ for $N=m+k+n+1$. Let $\lambda_1,\ldots,\lambda_{|V|}$ denote the multiset of (possibly complex) eigenvalues of $A$, listed according to algebraic multiplicity. Then $\lambda_1^N,\ldots,\lambda_{|V|}^N$ comprise the multiset of eigenvalues of $A^N$. But since $\rank(A^N)\leq 1$, there can be at most 1 nonzero eigenvalue.

      In fact, all eigenvalues must vanish. To see why, consider $\Tr A=\sum_{i\in V}w(i,i)$. Since $(i,i,i)$ is a directed triangle, our hypothesis that $G$ lacks directed triangles implies that $w(i,i)=0$ and thus $\Tr A=0$. Combined with the previous paragraph, this implies that all eigenvalues vanish and thus $A$ is nilpotent.

      From this it follows that for all but finitely many words $x\in \bigcup_{\ell\geq 0} V^\ell$ we must have $w(x)=0$. This implies that $G$ is not recurrent, contradicting the assumption that $G$ satisfies property (EC).
		\end{proof}

		A weighted digraph $G$ is {\bf strongly connected} if for every pair of
    vertices $i,j\in V$, there is a positive-weight word $x\in \bigcup_{\ell\geq 1}V^\ell$ that begins with $i$ and ends with $j$. We denote the existence of such an $x$ by writing $i\to j$. The
    {\bf strongly connected components} (SCCs) of a weighted digraph are the
    equivalence classes of the relation
    \[
      \{(i,j)\in V^2\colon i\to j\text{ and }j\to i\}.
    \]
    Note that $i\to i$ because singleton words are defined to have weight 1 (see the remark prior to \cref{dfn:B}). Also, recall from our discussion following \cref{dfn:B} that a word has positive weight if and only if $w(x)>0$, which occurs if and only if $B(x)>0$.

    Recall that a weighted digraph is a function $w\colon V^2\to \m R_{\geq 0}$. A subset $U\subseteq V$ induces a subdigraph by restriction of $w$ to $U^2$.

    Say that an SCC is {\bf recurrent} if the induced subdigraph of $G$ is recurrent. By our remarks at the beginning of \cref{subsection:weightedInsert}, an SCC with vertex set $\c C$ is recurrent if and only if there are arbitrarily large $n$ such that there exists a positive-weight word $x\in \c C^n$.

    Note that any SCC with at least two vertices is recurrent, for if $a$ and $b$ are any two such vertices, then each of the words $(ab)^n\in \c C^{2n}$ have positive weight. Also observe that for a directed acyclic graph, the SCCs are singletons, none of which is recurrent. Moreover, a singleton SCC is recurrent if and only if its vertex has a self-loop. In general, each directed cycle is contained in a recurrent SCC.

    For the remainder of this section, we restrict attention to weighted digraphs satisfying property (EC) whose associated insertion process is finitely dependent. We will show in Lemma \ref{lem:restrictiveStructure} that such digraphs are not too far from being strongly connected.


    \begin{lemma}\label{lem:restrictiveStructure}
      Let $G$ be a weighted digraph with property (EC) whose associated insertion process is finitely dependent.
      \begin{enumerate}
      \item $G$ has a
      unique recurrent SCC.
      \item Vertices not in the recurrent SCC belong to
      singleton SCCs that have no directed path to or from the recurrent SCC.
      \item The subdigraph of $G$ induced by the recurrent SCC also has property (EC). Moreover, its associated insertion process coincides with that of $G$.
      \item If moreover $G$ has property (C), then it is strongly connected.
      \end{enumerate}
    \end{lemma}

    In particular, it follows from Lemma \ref{lem:restrictiveStructure} that the family of finitely dependent insertion processes associated to weighted digraphs is unchanged if one restricts attention to strongly connected weighted digraphs.

    \begin{proof}
      Let $(X_i)_{i\in\m Z}$ denote the insertion process on $G$. Suppose without loss of generality that it is $(k+1)$-dependent, for some $k\geq 0$. In particular, we have the i.i.d.\ subsequence $(X_{ik})_{i\in\m Z}$.

      Let $\c C_1,\ldots,\c C_\ell$ denote the strongly connected components of $G$. The weighted digraph $G$ induces the weighted digraph $G_{\SCC}$ on the set of SCCs, denoted $V_{\SCC}:=\{\c C_1,\ldots,\c C_\ell\}$, given by the weight
      $$
        w_{\SCC}(\c C_i,\c C_j):=\sum_{u\in \c C_i,v\in \c C_j}w(u,v).
      $$
      Observe that $\c C_i\to\c C_j$ in $G_{\SCC}$ if and only if for some $u\in \c C_i$ and $v\in \c C_j$, there is a positive-weight word $x$ composed of vertices of $G$ that begins with $u$ and ends with $v$. Thus if $\c C_i\to \c C_j$ and $\c C_j\to \c C_i$, then $\c C_i\cup \c C_j$ is also an equivalence class, whereupon $\c C_i=\c C_j$.

      Consider the function $f\colon V\to V_{\SCC}$ that assigns to each vertex its strongly connected component. Then the sequence $\bigl(f(X_i)\bigr)_{i\in\m Z}$ has the property that for all $i<j$, we have $f(X_i)\to f(X_j)$ in $G_{\SCC}$ almost surely. But if $f(X_i)$ occurs once in the sequence, it a.s. occurs infinitely many times (by passing to an i.i.d.\ subsequence). Thus there exists $\ell>j$ such that $f(X_\ell)=f(X_i)$. Consequently the sequence
      \begin{equation}\label{eq:isConstant}
        \bigl(f(X_i)\bigr)_{i\in\m Z}=(\ldots,\c C_{\rec},\c C_{\rec},\ldots)
      \end{equation}
      is a.s. constant. Furthermore, the unique value $\c C_{\rec}$ that it takes must be a recurrent SCC, since the marginals of the process $(X_i)_{i\in \m Z}$ then yield arbitrarily long positive-weight words in $\bigcup_{\ell\geq 1}\c C_{\rec}^\ell$.

      On the other hand, each vertex $v\in V$ belonging to a recurrent SCC occurs infinitely often in the sequence $(X_i)_{i\in\m Z}$. Indeed, we can fix a sufficiently long positive-weight word containing $v$ and observe that it occurs with positive density in $(X_i)_{i\in\m Z}$ by finite dependence. Thus by the previous paragraph, it follows that $\c C_{\rec}$ is the a unique recurrent SCC of $G$, establishing property (i).

      Property (ii) follows by combining our previous observation that an SCC containing distinct vertices is recurrent with the following argument. If there was a directed path joining $\c C_{\rec}$ to another strongly connected component denoted $\c C'$, then there would be arbitrary long positive-weight words that contain vertices in $\c C'$. But then $\c C'$ appears in the sequence $\bigl(f(X_i)\bigr)_{i\in\m Z}$, by consideration of a sufficiently long marginal of the insertion process. This contradicts \eqref{eq:isConstant}, proving (ii).

      For all $n>|V|$, every word $x\in V^n$ contains some vertex at least twice. Thus by (ii), it follows that if such a word satisfies $B(x)>0$, then necessarily $x\in \c C_{\rec}^n$. Hence
      $$
        \sum_{x\in V^n}B(x)=\sum_{x\in \c C_{\rec}^n}B(x),
      $$
      and it follows that all sufficiently long marginals of the insertion processes associated to $G$ and the subdigraph induced by $\c C_{\rec}$ coincide. This yields (iii).

      Finally, (iv) follows from the observation that $B(x)=1$ for all words of length 1. Thus if $G$ satisfies property (C), the random variable $X_0$ in the associated insertion process is uniformly random on the vertex set $V$. But by \eqref{eq:isConstant} we have that $X_0\in \c C_{\rec}$ almost surely, implying that $V=\c C_{\rec}$ and therefore $G$ is strongly connected.
    \end{proof}

\section{Uniform Weight Graphs}\label{section:unifWeight}
	We establish \cref{thm:second} in this section, and deduce \cref{thm:main}
  as a corollary. The plan of attack is as follows:
	\begin{enumerate}
		\item \Cref{subsection:propUnif} defines the special class of weighted digraphs to which \cref{thm:second} applies, which we call uniform weight graphs. We deduce necessary structural properties for a uniform weight graph to satisfy property (EC).
		\item \Cref{subsection:modClique} is devoted to complete multipartite
    graphs of uniform weight. It is shown that understanding finite dependence for such graphs reduces to the analysis of complete graphs, which has already been undertaken in \cite{holroyd60}.
		\item \Cref{subsection:uniqueUnif} proves \cref{thm:second} by using a combinatorial argument to show that the structural properties in \Cref{subsection:propUnif} imply that the graph is complete multipartite, then applying the results of \Cref{subsection:modClique}.
	\end{enumerate}
		\subsection{Uniform weights}
    \label{subsection:propUnif}

		Uniform weight graphs are weighted digraphs that are undirected and of constant weight. Their associated insertion processes are sufficiently general to encompass the end-weighted insertion processes of \cite[Section 4]{holroyd63} as a special case.

		\begin{samepage}
		\begin{dfn}\label{dfn:unifWeight}
      A {\bf uniform weight graph} with weight $w>0$ is a weighted digraph whose weight function satisfies that for all $i,j\in V$,
			\begin{enumerate}
				\item either $w(i,j)=0$ or $w(i,j)=w$, and
				\item $w(i,j)=w(j,i)$, and
				\item $w(i,i)=0$.
			\end{enumerate}
		\end{dfn}
		\end{samepage}
    
    We regard uniform weight graphs as being undirected. As such, we refer to their strongly connected components simply as connected components (when applying \cref{lem:restrictiveStructure}, for instance). For uniform weight graphs, we say that vertices $i$ and $j$ are {\bf adjacent} if and only if $w(i,j)>0$, in which case $w(i,j)=w$.

    The recurrence \eqref{eqn:propCdfn} simplifies to the following in the case of a positive-weight word in a uniform weight graph
    \begin{equation}\label{eq:recUnif}
      B(x)=wB(\hat x_1) + w^2\sum_{i=2}^{n-1} B(\hat x_i) + wB(\hat x_n),\qquad x\in V^n.
    \end{equation}

    It is convenient to fix some notation regarding alternating words. For any positive integer $n$ that may be even or odd, we define $(ab)^{n/2}$ to be the unique alternating word in $\{a,b\}^n$ beginning with $a$.

		\begin{lemma}\label{lem:calculation} If $w(a,b)>0$ and $w(b,v)>0$, then
			\begin{equation}\label{eq:altCalc}
        B\left((ab)^{n/2}\ v\right)=\begin{cases}
          (2w)^n,&w(a,v)=0\\
          \frac{2w^2(w+w^2)^{n-1}-(2w)^n}{w-1},&w(a,v)=w\text{ and }w\not=1\\
          2^{n-1}(n+1),&w(a,v)=w\text{ and }w=1.
        \end{cases}
      \end{equation}
		\end{lemma}
		\begin{proof}
      First suppose that $w(a,v)=0$. By \eqref{eq:recUnif} we have that
      \[
        B\left((ab)^{n/2}\ v\right)=w\cdot B\left((ba\right)^{(n-1)/2}\ v)+
        w\cdot B\left((ab)^{n/2}\right).
      \]
      Applying this inductively yields the first case of \eqref{eq:altCalc}.

      Next, suppose that $w(a,v)=w$ and $w\not=1$. Again by \eqref{eq:recUnif},
      \begin{align*}
        B\left((ab)^{n/2}\ v\right)&=w\cdot B\left((ba)^{(n-1)/2}\ v\right)+
        w^2\cdot B\left((ab)^{(n-1)/2}\ v\right)+
        w\cdot B\left((ab)^{n/2}\right)\\
        &=(w+w^2)\cdot B\left((ba)^{(n-1)/2}\ v\right)+w\cdot (2w)^{n-1}.
      \end{align*}
      Now the second case of \eqref{eq:altCalc} follows by induction, starting from the base case $B(av)=2w$. The final case of \eqref{eq:altCalc} follows from a simplification of the previous calculation.
		\end{proof}

    \begin{lemma}\label{lem:propUnif}
      If $G$ is a non-empty connected uniform weight graph satisfying property (EC), then the
      following conditions hold.
      \begin{enumerate}
      \item
      The graph $G$ is $d$-regular for some $d\geq 1$.
      \item
      For some $t\geq 0$, there are $t$ triangles on every edge of $G$.
      \end{enumerate}
    \end{lemma}
    \begin{proof}
      Let $w$ be the common weight of the edges in $G$. Fix a vertex $b$. Let $d$ denote the degree of $b$. By hypothesis, there exists an edge $(a,b)$. Let $t$ denote the number of triangles on the edge $(a,b)$. First suppose that $w\not=1$. By \Cref{lem:calculation},
      \[
        \sum_{v}B\left((ab)^{n/2}\ v\right)=(2w)^{n}(d-t)+
        \left[\frac{2w^2(w+w^2)^{n-1}-(2w)^n}{w-1}\right]t.
      \]
      Recall that $B\bigl((ab)^{n/2}\bigr)=(2w)^{n-1}$. Thus by \eqref{eqn:propCdfn}, it follows that
      \begin{equation}\label{eqn:eqn_d_t_Cn}
        C_n=2w(d-t)+\left[\frac{w(\frac{w+1}{2})^{n-1}-1}{w-1}\right](2wt)
      \end{equation}
      for all sufficiently large $n$ if $w\not=1$. Likewise if $w=1$, for all sufficiently large $n$ we have that $C_n=2(d-t)+(n+1)t.$ In either case, both $d$ and $t$ are determined by the values of $C_n$ for $n$ sufficiently large. Thus $d$ and $t$ take the same value, for all vertices $b$ and all edges $(a,b)$.%
    \end{proof}

		Consider the graph that appears in \cref{fig:coat}.

		\begin{dfn}\label{dfn:kite}
			An undirected graph on the vertices $a,b,c,d$ is a {\bf kite} if $(a,b,c)$
      form a triangle and $d$ is adjacent to $a$ and to no other vertex.
		\end{dfn}
		\begin{figure}
			\centering
			\includegraphics{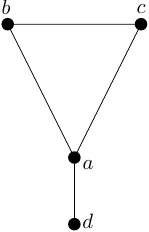}
			\caption{The kite $(abc;d)$\label{fig:coat}}
		\end{figure}
		\begin{lemma}\label{lem:dt}
			Suppose that $G$ has uniform weight and satisfies property (EC). Then
      there is no kite which occurs as an induced subgraph of $G$.
		\end{lemma}
		\begin{proof}
			Suppose to the contrary that $G$ contains a kite $(abc;d)$. For all
      $n\geq 3$, we will construct words $x$ and $y$ satisfying
			\[
				\sum_{v\in V}\frac{B(xv)}{B(x)}>\sum_{v\in V}\frac{B(yv)}{B(y)},
			\]
			from which it will follow that $G$ cannot satisfy property (EC).

      Let $x=(ba)^{n/2}$ and let $y=d (ab)^{(n-1)/2}$. Simple casework reveals
      that for each $\sigma\in S_n$ we have $w(x;\sigma)=w(y;\sigma)$, and
      consequently $B(x)=B(y)$. Note that this common value is positive
      (by \cref{lem:calculation}, for instance). Thus we need only establish
      that \[\sum_{v}B(xv)>\sum_v B(yv).\] First observe that we have the weaker
      inequality
      \begin{equation}\label{eq:weak}
        \sum_{v}B(xv)\geq \sum_v B(yv).
      \end{equation}
      Indeed, each $\sigma\in S_n+1$ with $w(yv;\sigma)>0$ is seen to satisfy
      $w(xv;\sigma)>0$ as well. Since the graph has uniform weight, it
      therefore follows that $w(xv;\sigma)\geq w(yv;\sigma)$ and the claim
      follows upon summation.

			To obtain strict inequality, observe that if $\sigma\in S_{n+1}$ satisfies $\sigma(\{1,2\})=\{1,n+1\}$ then $w(xc;\sigma)>w(yc;\sigma)=0$. When combined with the previous inequality, it follows that $B(xc)>B(yc)$.
		\end{proof}

		\subsection{Complete multipartite graphs}\label{subsection:modClique}
		In this section we relate property (EC) for complete
    multipartite graphs to property (EC) for complete graphs. This allows us to
    extend results of Holroyd and Liggett \cite{holroyd60} to handle complete
    multipartite graphs.

		We use the following notation: $K_q$ denotes the complete graph on
    $q$ vertices, $K_{r,\ldots,r}$ denotes the complete multipartite graph
    with $q$ parts each of size $r$, and $w\cdot K_{r,\ldots,r}$ denotes the
    uniform weight graph in which each edge of the corresponding complete multipartite graph has weight $w>0$. In all cases, we take the vertex set to be $[qr]:=\{1,\ldots,qr\}$ (with $r=1$ in the case of $K_q$), and we take the edge set to be $\{(i,j)\colon i\not\equiv j\mod q\}$. In the case of $w\cdot K_{r,\ldots,r}$, the weight function is given by $w(i,j)=w\cdot\one{i\not\equiv j\mod q}$.


		Note that the graph $K_{r,\ldots,r}$ is a Tur\'an graph
    \cite{bollobas1998modern}.

		\begin{lemma}\label{prop:turan}
      For any $r\geq 1$ and $k\geq 0$, the graph $w\cdot K_{r,\ldots,r}$ satisfies property (EC) and its associated insertion process is $k$-dependent if and only if the same holds for $w\cdot K_q$.
		\end{lemma}
		\begin{proof}
			Consider the mapping $f\colon [qr]\to [q]$ given by $f(i)=i\mod q$. For
      any word $x\in [qr]^n$, let $f(x)$ denote the word
      $f(x_1)f(x_2)\cdots f(x_n)\in [q]^n$. Observe that
			\[w(i,j)=w\bigl(f(i),f(j)\bigr)\quad \forall i,j\in [qr],\]
			from which it follows that $B(x)=B\bigl(f(x)\bigr)$. Thus the
      following are equivalent:

			\begin{enumerate}
				\item $\ds\sum_{v\in [qr]}B(xv)=rC_{n}B(x)$
				\item $\ds\sum_{v'\in [q]}B\bigl(f(x) v'\bigr)=C_{n} B\bigl(f(x)\bigr)$
			\end{enumerate}
			By \eqref{eqn:propCdfn}, it follows that $w\cdot K_q$ satisfies property
      (EC) if and only if $w\cdot K_{r,\ldots,r}$ satisfies property (EC).

			Similarly, the following are equivalent:
			\begin{enumerate}
				\item $\ds \sum_{W\in [qr]^k}B(x\; W\; y)=r^kC_{nm}B(x)B(y)$
				\item $\ds \sum_{W'\in [q]^k}B\bigl(f(x)\; W'\; f(y)\bigr)=
        C_{nm} B\bigl(f(x)\bigr) B\bigl(f(y)\bigr)$
			\end{enumerate}
			The result now follows by \cref{lem:kDep}.
		\end{proof}

    \begin{lemma}\label{lem:wIs1}
      Suppose that $w\not=1$ and $q\geq 3$. Then the graph $w\cdot K_q$ does
      not satisfy property (EC).
    \end{lemma}
    \begin{proof}
    Consider an integer $n\geq 3$. Fix distinct vertices $a,b,c\in K_q$.
    Keeping in mind our notation for alternating words, set
      \[
        x=(ba)^{n/2}\in \{a,b\}^n
        ,\qquad y=c\ (ab)^{(n-1)/2}\in \{a,b,c\}^n
      \]
      (regardless of the parity of $n$). Consider the quantities
      \[
        Q_n=\sum_{v\in K_q} \frac{B(xv)}{B(x)},\qquad R_n=\sum_v
        \frac{B(yv)}{B(y)}.
      \]
      We will prove by induction on $n$ that $Q_n>R_n$ for all $w>1$ and
      $Q_n<R_n$ for all $w\in (0,1)$.

      From \eqref{eq:recUnif} we deduce the recurrences $B(xv)=wB(\hat x_1v)+w^2B(\hat x_nv)+wB(x)$ and
      $B(yv)=wB(\hat y_1v)+w^2B(\hat y_2v)+w^2B(\hat y_nv)+wB(y)$, yielding that
      \begin{align*}
        Q_n&=\left(\frac{w+1}{2}\right)Q_{n-1}+qw-w(w+1)\\
        R_n&=\frac{wB(\hat x_1)}{B(y)}Q_{n-1}+\frac{2w^2
        B(\hat y_2)}{B(y)}R_{n-1}+qw-w(w+1).
      \end{align*}

      Subtracting the previous equations and substituting the recurrence
      $B(y)=wB(\hat y_1)+w^2B(\hat y_2)+wB(\hat y_n)$ yields that $2B(y)(Q_n-R_n)/w$ equals
      \begin{equation}\label{eqn:recIneq}
        \left[(w-1)B(\hat x_1)+(w+1)^2B(\hat y_2)\right]Q_{n-1}-
        4wB(\hat y_2)R_{n-1}
      \end{equation}
      When $w>1$, we leave out the first term, using $(w+1)^2\geq 4w$ to obtain
      \[
        2B(y)(Q_n-R_n)/w>4w(Q_{n-1}-R_{n-1}).
      \]
      Since the right side vanishes for $n=3$, it follows that $Q_n>R_n$ by
      induction.

      Next suppose that $w<1$. This case requires a tighter bound.
      We begin by establishing that for all $m\geq 3$,
      \[
        B\left((ba)^{m/2}\right)>(1-w)B\left(c\ (ba)^{(m-1)/2}\right).
      \]
      Indeed, we deduce the bound inductively from
      \begin{align*}
        B\left((ba)^{m/2}\right)&=2wB\left((ab)^{(m-1)/2}\right)\\
        &=[(1-w)w+w^2+w]B\left((ab)^{(m-1)/2}\right)\\
        (\text{induction})\qquad &> (1-w)\biggl[w B\left((ab)^{(m-1)/2}\right)
        + w^2 B\left(c\ (ba)^{(m-2)/2}\right)\\
        &\qquad  + w B\left(c\ (ba)^{(m-2)/2}\right)\biggr]\\
        &= (1-w)B\left(c\ (ab)^{(m-1)/2}\right).
      \end{align*}
      Consequently
      \begin{equation}\label{eqn:auxBound}
        (1-w)B(\hat x_1)>(1-w)^2B(\hat y_2)
      \end{equation}
      \noindent Now we plug the bound \eqref{eqn:auxBound} into
      \eqref{eqn:recIneq} to obtain a bound which is suitable for induction:
      \begin{align*}
      2B(y)(R_n-Q_n)/w&=4wB(\hat y_2)R_{n-1}\\
      &\qquad +
\left[(1-w)B(\hat x_1)-(w+1)^2B(\hat y_2)\right]Q_{n-1}\\
\text{using \eqref{eqn:auxBound}}\qquad &>4wB(\hat y_2)(R_{n-1}-Q_{n-1}).
      \end{align*}
      As before, when $n=3$ the right side vanishes. Thus it follows by
      induction that if $0<w<1$, we have $R_n>Q_n$ for all $n\geq 3$.
      Combined with the previous case, we conclude that when $w\not=1$ and
      $q\geq 3$, the graph $w\cdot K_q$
      does not satisfy property (EC).
    \end{proof}

		Combining several results allows us to determine which uniform weight complete multipartite graphs satisfy property (EC) and have a finitely dependent associated insertion process.
		\begin{lemma}\label{lem:stringing}
			The graph $w\cdot K_{r,\ldots,r}$ for $q\geq 3$ and $r\geq 1$ has
      property (EC) and has a $k$-dependent associated insertion process if and
      only if $w=1$ and either: $q=3$ and $k\geq 2$; or $q=4$ and $k\geq 1$.
		\end{lemma}
		\begin{proof}
			By \cref{prop:turan}, it suffices to consider $w\cdot K_q$. By
      \cref{lem:wIs1}, $w=1$. Applying Propositions 10 and 13 of
      \cite{holroyd60}, the insertion process on $K_q$ is $k$-dependent if
      and only if $q=3$ ($k\geq 2$) or $q=4$ ($k\geq 1$).
		\end{proof}
	\subsection{Extension to uniform weight graphs}\label{subsection:uniqueUnif}
		We now combine the results of
    \Cref{subsection:modClique,subsection:propUnif} to extend
    the conclusion of \cref{lem:stringing} to all uniform weight graphs.
    We establish that the only uniform weight graphs satisfying property (EC)
    that have a $k$-dependent associated
    insertion process are $K_{r,r,r}$ and $K_{r,r,r,r}$, and unions thereof with isolated vertices.
    Note that these graphs have weight $w=1$, even though we allow $w>0$ to be arbitrary a priori. These graphs
    are closely related to the graphs $K_3$ and $K_4$ corresponding to the
    colorings discovered by Holroyd and Liggett; in fact, the graphs $K_{r,r,r}$
    and $K_{r,r,r,r}$ are obtained from $K_3$ (resp. $K_4$) by replacing each
    vertex with $r$ copies of itself. Similarly, the insertion process on $K_{r,r,r}$ (resp. $K_{r,r,r,r}$) is obtained from the corresponding process on $K_3$ (resp. $K_4$) by replacing each instance of vertex $i$ with an i.i.d.\ choice of one of its $r$ copies in $K_{r,r,r}$ (resp. $K_{r,r,r,r}$).

		The following graph-theoretic lemma allows us to reduce the general
    uniform weight case to that of the complete multipartite graphs treated in
    \Cref{subsection:modClique}. We will use \cref{lem:kiteFree} to show that
    any uniform weight graph either contains a kite, or is complete
    multipartite.

		\begin{lemma}\label{lem:kiteFree}
			Let $G$ be a kite-free connected loopless graph containing a triangle
      $abc$. Then every vertex $d\in V$ is adjacent to at
      least two of $\{a,b,c\}$.
		\end{lemma}

		\begin{figure}[ht]
			\centering
			\includegraphics{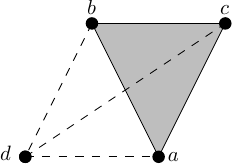}
			\caption{The problematic triangle $abc$ in \cref{lem:kiteFree}
      \label{fig:badTri}}
		\end{figure}
		\begin{proof}
			By \cref{lem:dt}, no vertex $d\in V$ is adjacent to exactly one of
      $abc$. Hence it suffices to show that every vertex $d$ is adjacent to
      $abc$.

			Suppose to the contrary that some $d\in V$ is non-adjacent to $abc$
      (\cref{fig:badTri}). Choose a minimal path joining $d$ to $abc$. We show
      that a kite is present near the intersection of the path with $abc$,
      from which we obtain the desired contradiction.

			\begin{figure}[ht]
				\centering
				\includegraphics{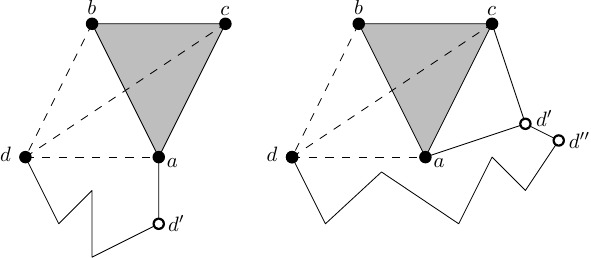}
				\caption{Reaching a contradiction\label{fig:pathToTri}}
			\end{figure}

			Let $d'$ denote the path vertex adjacent to $abc$. There are two cases
      to consider: either $d'$ is adjacent to a single vertex of $abc$
      (left half of \cref{fig:pathToTri}), or it is adjacent to more than
      one vertex (right half of \cref{fig:pathToTri}).

			If $d'$ is adjacent to a single vertex, then $(abc;d')$ is a kite. Now
      suppose that $d'$ is adjacent to more than one vertex. Without loss of
      generality, suppose that $d'$ is adjacent to $a$ and $c$. Let $d''$
      denote a neighbor of $d'$ on the path from $d$ to $d'$. By minimality
      of the path, $d''$ is non-adjacent to $abc$. Consequently $(d'ac;d'')$
      is a kite.
		\end{proof}

		We are now in a position to prove \cref{thm:second}.
		\begin{proof}[Proof of \cref{thm:second}]
			Suppose that $G$ is a uniform weight graph with property (EC)
      and suppose
      that the insertion process associated to $G$ is $k$-dependent. We will
      deduce that either $G=K_{r,r,r,r}$ and $k\geq 2$, or $G=K_{r,r,r}$ and
      $k\geq 1$, or $G$ is a disjoint union of one of these graphs with a
      collection of isolated vertices.

      Since uniform weight graphs are undirected,
      \cref{lem:restrictiveStructure} takes on a simpler form in the present
      context. Indeed, it implies that both the (EC) property and the insertion
      process are unchanged by deletion of isolated vertices, and moreover it
      implies that the resulting graph is connected. Hence it suffices to
      consider the connected case.

			We show in this case that $G$ is complete multipartite. Consider the
      relation
      \begin{equation}\label{eq:rel}
        \{(i,j)\in V^2\colon w(i,j)=0\}.
      \end{equation}
      This relation is reflexive and symmetric by definition of a uniform weight graph. Once we establish transitivity,
      it will follow that the graph $G$ is complete multipartite with
      partite sets given by the equivalence classes of this relation.

			Suppose to the contrary that transitivity did not hold. Then there would
      exist vertices $a,b,d$ such that $w(a,b)>0$ yet $w(a,d)=w(b,d)=0$. By \cref{thm:lackTri} and \cref{lem:propUnif}, there are
      $t\geq 1$
      triangles on every edge. In particular, we may complete the edge
      $(a,b)$ into a triangle $abc$. By \cref{lem:dt}, $G$ lacks kites,
      and by assumption it is connected. Moreover by definition of uniform weight, $G$ is loopless.

			Thus we have verified the conditions of \cref{lem:kiteFree}.
      Consequently, the vertex $d$ is adjacent to at least two of $a,b,c$.
      This contradicts the hypothesis that $w(a,d)=w(a,b)=0$. Thus the relation \eqref{eq:rel} is transitive, and we deduce that it is an equivalence relation.

			Decompose the vertex set into equivalence classes of \eqref{eq:rel}. Then $G$ is complete multipartite, with partite sets are given by the equivalence classes of \eqref{eq:rel}. By \cref{lem:propUnif}, $G$ is a regular graph and thus the parts have
      equal sizes. Thus $G=w\cdot K_{r,\ldots,r}$, for some $w>0$ and $r\geq 1$.

			Applying \cref{thm:lackTri} again, we see that $G$ contains a triangle
      and therefore $q\geq 3$. The result now follows when $G$ has property (EC) by \cref{lem:stringing}.

      Finally, observe that if $G$ also satisfies property (C), then there can be no isolated vertices by \cref{lem:restrictiveStructure}(iv).
		\end{proof}

\section{Complete weighted digraphs with property (C)}
  \label{section:arbWeight}
	The results in this section apply to loopless complete weighted digraphs satisfying property (C). That is, the digraphs under consideration satisfy ${w(i,i)=0}$ and ${w(i,j)>0}$ for all distinct vertices $i$ and $j$, as well as property (C). We will establish \cref{thm:third}, which states that the only
  such graphs for which the associated insertion process is $k$-dependent are
  the (unweighted) graphs $K_3$ (for $k\geq 1$) and $K_4$ (for $k\geq 2$). We
  will establish this result by reducing it to the case of a
  uniform weight graph and applying results from \Cref{section:unifWeight}.

	For distinct indices $i,j\in V$, we define the quantity
	\[
		T_n(i,j)=\sum_{v\in V} w(i,v)^{\ceil{n/2}}w(j,v)^{\fl{n/2}},
	\]
  where we use the convention $0^0:=1$.
	\begin{lemma}\label{lem:invariant}
    Fix a loopless complete weighted digraph satisfying property (C) and fix an integer $n\geq 1$. Then the value of $T_n(i,j)$ is constant over all vertices $i$ and $j$ with $w(i,j)>0$.
	\end{lemma}

	\begin{proof}
		Let $x\in \{i,j\}^n$ be the unique alternating word ending in $i$. Recall from \cref{dfn:B} that
    \begin{equation}\label{eq:sumOfSigma}
      B(xv)=\sum_{\sigma\in S_n}w(xv;\sigma).
    \end{equation}
    Since $w(i,i)=w(j,j)=0$ by assumption, for all permutations with $w(xv;\sigma)>0$ we have that
		\begin{equation}\label{eq:sigmaBuildProd}
			w(xv;\sigma)=w(x) w(i,v)^{\ceil{\ell/2}}w(j,v)^{\fl{\ell/2}},
		\end{equation}
		where $\ell$ is given by the formula
    \begin{equation}\label{eq:ellFormula}
       \ell=n+1-\max_{t<\sigma^{-1}(n+1)} \sigma(t).
    \end{equation}
    This is straightforward to verify from the definitions and \Cref{fig:pruned}.


		\begin{figure}
			\centering
			\def\svgwidth{0.5\textwidth}
			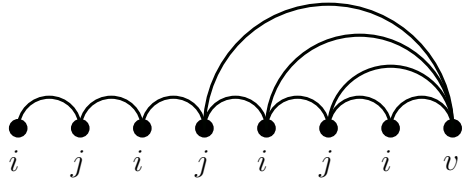
      \vspace{1.5mm}
			\caption{Let $x$ be as in the proof of \Cref{lem:invariant}. For $\sigma\in S_{n+1}$, let $\ell$ be as in \eqref{eq:ellFormula}. Consider the nearest neighbor graph on $\{1,\ldots,n+1\}$, except we modify vertex $n+1$ to have degree $\ell$. This graph is drawn above with vertex $i$ labeled with the $i^{th}$ symbol of the word $xv$. Then $w(xv;\sigma)=\prod_{e=(i,j)}w\bigl((xv)_i,(xv)_j\bigr)$, where the product is taken over edges of the graph above.}\label{fig:pruned}
		\end{figure}

		By \eqref{eq:ellFormula}, we see that $\ell$ ranges over $\{1,\ldots,n\}$ as $\sigma$ ranges over $S_{n+1}$. Thus substituting \eqref{eq:sigmaBuildProd} into \eqref{eq:sumOfSigma} implies that there are integers $d_1,\ldots,d_n>0$ whose values depend only on $n$ such that 
		\al{
			B(xv)&=w(x)\sum_{\ell=1}^n d_{\ell}\;
      w(i,v)^{\ceil{\ell/2}}w(j,v)^{\fl{\ell/2}}.
		}
		Summing over all vertices $v\in V$, we obtain
		\al{
			\sum_{v\in V}B(xv)&=w(x)\sum_{\ell=1}^nd_{\ell}\; T_{\ell}(i,j).
		}

		By \eqref{eqn:propCdfn} we have that $\sum_{v\in V}B(xv)=C_nB(x).$ Furthermore $B(x)=2^{n-1}w(x)$, by \cref{lem:lackTri} applied to the subgraph induced by $\{i,j\}$. As $w(x)>0$, we have that
		\[
			\sum_{\ell=1}^nd_{\ell}\; T_{\ell}(i,j)=2^{n-1}C_n.
		\]
    Since $d_1,\ldots,d_n>0$, we may explicitly solve this system of equations to obtain that
    $$
      \begin{pmatrix}T_1(i,j)\\\vdots\\T_n(i,j)
      \end{pmatrix}=
      \begin{pmatrix}
        d_1 & 0 & \cdots & 0\\
        d_1 & d_2 & \cdots & 0\\
        d_1 & d_2 & \ddots & 0\\
        d_1 & d_2 & \cdots & d_n\\
      \end{pmatrix}^{-1}
      \begin{pmatrix}C_1\\2C_2\\\vdots\\2^{n-1}C_n
      \end{pmatrix}.
    $$
		In particular, it follows that $T_n(i,j)$ is independent of the pair $(i,j)$.
	\end{proof}
	In light of \cref{lem:invariant}, we write $T_n$ in place of $T_n(i,j)$
  from now on.
	\begin{lemma}\label{lem:offdiagUnif}
		For all pairs of distinct indices $(i,j)$ and $(i',j')$, we have
    $w(i,j)=w(i',j')$.
	\end{lemma}
	\begin{proof}
		Let $z_1>z_2>\cdots $ denote the set of distinct positive values attained
    by $w(i,v)w(j,v)$ as $v$ ranges over $V$. Let $V_{\ell}$ denote the
    vertices which contribute to $z_{\ell}$, given by
		\[
			V_{\ell}=\{v\in V\colon w(i,v)w(j,v)=z_{\ell}\},
		\]
		and let $a_{\ell}=\sum_{v\in V_{\ell}}w(i,v)$. Rewriting the expression
    for $T_{2n+1}$ yields
		\[
			T_{2n+1}=\sum_{\ell}a_{\ell}z_{\ell}^n.
		\]
		Next, observe that $z_1=\inf_{n\to\infty}(T_{2n+1})^{1/n}$ and $a_1=\inf_{n\to\infty}\frac{T_{2n+1}}{z_1^n}.$ Hence the parameters $a_1$ and $z_1$ can be reconstructed given the
    sequence $\{T_n\}$. Applying the same procedure to $T_{2n+1}-a_1z_1^n$
    allows us to iteratively reconstruct all of the parameters. Thus
    $a_{\ell}$ and $z_{\ell}$ are uniquely determined by the $\{T_n\}$, so
    they are independent of the choice $(i,j)$. Finally, observe that $\sum_{\ell}a_{\ell}=\sum_{v\not=j}w(i,v)$. Therefore
		\[
			w(i,j)=T_1-\sum_{\ell}a_{\ell},
		\]
		which shows that for $i\not=j$, the value of $w(i,j)$ is constant.
	\end{proof}

	Combining \cref{lem:offdiagUnif} with the results of
  \Cref{section:unifWeight} allows us to conclude \cref{thm:third}.

	\begin{proof}[Proof of \cref{thm:third}]
		Suppose that $G$ is a loopless weighted digraph such that $w(i,j)>0$ for
    all distinct vertices $i$ and $j$. Moreover, suppose that $G$ has
    property (C) and that the associated insertion process is $k$-dependent
    process. We show that either: $G=K_3$ and $k\geq 2$; or $G=K_4$ and
    $k\geq 1$.

		By \cref{lem:offdiagUnif}, the graph $G$ has uniform weight, and by
    \cref{lem:restrictiveStructure} it is strongly connected.
    Applying \cref{thm:second}, we conclude that either $G=K_{r,r,r}$
    and $k\geq 2$, or $G=K_{r,r,r,r}$ and $k\geq 1$. Since $w(i,j)>0$ for
    all $i\not=j$, it follows that $r=1$. Thus $G$ is a complete graph, and
    applying \cite[Proposition 13]{holroyd60} we deduce that $G$ is either
    $K_3$ (for $k\geq 2$) or $K_4$ (for $k\geq 1$).

		Conversely, by the main result of \cite{holroyd60} the graphs $K_3$ and
    $K_4$ are $k$-dependent for $k\geq 2$ and $k\geq 1$ respectively.
	\end{proof}

	\section{Shifts of finite type}\label{section:shifts}
		We turn to the proof of \cref{thm:shiftFinite}, which states that if the de Bruijn graph of a loopless shift of finite type satisfies property (EC), then the associated insertion process on the shift of finite type is not finitely dependent.

		\begin{proof}[Proof of \cref{thm:shiftFinite}]
			Let $G$ denote the de Bruijn graph of the shift of finite type and let
      $\{Y_{\ell}\}_{\ell\in \m Z}$ denote the insertion process associated
      to $G$. We write
			\[Y_{\ell}=\left(x_{\ell},\ldots,x_{\ell+n-1}\right)\in [q]^n.\]
			Since $\{Y_{\ell}\}$ is almost surely a path in $G$, the overlapping
      elements in adjacent tuples $Y_{\ell}$ and $Y_{\ell+1}$ almost surely
      coincide. We extend the insertion process from $G$ to the shift of
      finite type by considering the random sequence
      $\{x_{\ell}\}_{\ell\in \m Z}$.

      Suppose to the contrary that for some $k$, this process is $k$-dependent.
      Then the tuples $Y_{\ell}=\left(x_{\ell},\ldots,x_{\ell+n-1}\right)$
      form an $(n+k-1)$-dependent sequence. Applying \cref{thm:lackTri} to the
      insertion process on $G$, it follows that $G$ has a directed triangle
      $(a,b,c)$. Hence we may write
				\[
					a=(x_1,x_2,\ldots,x_n),\quad b=(x_2,\ldots,x_{n+1}),
          \quad c=(x_3,\ldots,x_{n+2}).
				\]
				Since the edge $(a,c)$ is present in $G$, we must have
				\[
					(x_2,x_3,\ldots,x_n)=(x_3,\ldots,x_{n+1}).
				\]
				Thus $x_2=x_3=\cdots=x_{n+1}$, so $b$ is a constant sequence.
        This contradicts our assumption that the vertex set of $G$ lacks
        elements of the form $(i,\ldots,i)$. Therefore the associated
        insertion process is not $k$-dependent.
		\end{proof}

	\section*{Acknowledgements}
		I thank Alexander Holroyd for posing the questions which inspired this
    paper. I am grateful to Alexander Holroyd and Christopher Hoffman for
    valuable discussions as well as their patience. Furthermore I thank the
    feedback of an anonymous referee for providing the inspiration both for the consideration of the (EC) condition and for \Cref{question:4}, in addition to numerous further helpful suggestions.

	\bibliographystyle{habbrv}
	\bibliography{coloring}

\end{document}